\documentclass[12pt]{amsart}
\usepackage{latexsym,amssymb,amsmath,amsthm,amscd,enumerate,graphicx,esint,longtable}

\makeatletter
\@namedef{subjclassname@2010}{%
  \textup{2010} Mathematics Subject Classification}
\makeatother

\newtheorem{thm}{Theorem}[section]
\newtheorem{cor}[thm]{Corollary}

\newtheorem{prob}[thm]{Problem}

\theoremstyle{definition}

\numberwithin{equation}{section}

\newcommand{\cQ}{{\mathcal Q}}
\newcommand{\cR}{{\mathcal R}}

\newcommand{\R}{{\mathbb R}}

\newcommand{\kM}{{\mathfrak M}}

\def\al{\alpha}

\def\dl{\delta}

\def\0{\emptyset}
\def\1{\textbf{\rm 1}}
\def\6{\partial}
\def\8{\infty}

\def\lt{\left}
\def\rt{\right}

\def\wt{\widetilde}

\DeclareMathOperator*{\essinf}{\,\text{ess\,inf}\,}

\begin{document}

\title
[The Fefferman-Stein type inequality]
{The Fefferman-Stein type inequality for strong maximal operator in the heigher dimensions}

\author[H.~Tanaka]{Hitoshi~Tanaka}
\address{
Research and Support Center on Higher Education for the hearing and Visually Impaired, 
National University Corporation Tsukuba University of Technology,
Kasuga 4-12-7, Tsukuba City, Ibaraki, 305-8521 Japan
}
\email{htanaka@k.tsukuba-tech.ac.jp}

\thanks{
The author is supported by 
Grant-in-Aid for Scientific Research (C) (15K04918), 
the Japan Society for the Promotion of Science. 
}



\date{}

\begin{abstract}
The Fefferman-Stein type inequality 
for strong maximal operator is verified 
with compositions of some maximal operators in the heigher dimensions. 
An elementary proof of the endpoint estimate for the strong maximal operator is also given. 
\end{abstract}

\maketitle

\section{Introduction}\label{sec1}
The purpose of this paper is 
to develop a theory of weights for strong maximal operator in the heigher dimensions 
and to present an elementary proof of the endpoint estimate for the strong maximal operator. 
We first fix some notations. 
By weights we will always mean non-negative and locally integrable functions on $\R^d$. 
Given a measurable set $E\subset\R^d$ and a weight $w$, 
$w(E)=\int_{E}w\,dx$, 
$|E|$ denotes the Lebesgue measure of $E$ and 
$\1_{E}$ denotes the characteristic function of $E$. 
Let $0<p\le\8$ and $w$ be a weight. 
We define the weighted Lebesgue space 
$L^p(\R^d,w)$ 
to be a Banach space equipped with the norm (or quasi norm) 
\[
\|f\|_{L^p(\R^d,w)}
=
\lt(\int_{\R^d}|f|^pw\,dx\rt)^{\frac1p}.
\]
For a locally integrable function $f$ on $\R^d$, 
we define the Hardy-Littlewood maximal operator $\kM$ by
\[
\kM f(x)
=
\sup_{Q\in\cQ}
\1_{Q}(x)\fint_{Q}|f|\,dy,
\]
where $\cQ$ is the set of all cubes in $\R^d$ 
with sides parallel to the coordinate axes and
the barred integral $\fint_{Q}f\,dy$ 
stands for the usual integral average of $f$ over $Q$. 
For a locally integrable function $f$ on $\R^d$, 
we define the strong maximal operator $\kM_d$ by
\[
\kM_df(x)
=
\sup_{R\in\cR_d}
\1_{R}(x)\fint_{R}|f|\,dy,
\]
where $\cR_d$ is the set of all rectangles in $\R^d$ 
with sides parallel to the coordinate axes. 

It is well known that 
\[
w(\{x\in\R^d:\,\kM f(x)>t\})
\le
\frac{C}{t}
\|f\|_{L^1(\R^d,\kM w)},
\quad t>0,
\]
holds for arbitrary weight $w$ and, 
by interpolation, that 
\[
\|\kM f\|_{L^p(\R^d,w)}
\le C
\|f\|_{L^p(\R^d,\kM w)},
\quad p>1,
\]
holds for arbitrary weight $w$. 
These are called the Fefferman-Stein inequality 
(see \cite{FS}). 

There is a problem in the book \cite[p472]{GR}:

\begin{prob}\label{prob1.1}
Does the analogue of the Fefferman-Stein inequality 
hold for the strong maximal operator, i.e. 
\begin{equation}\label{1.1}
\|\kM_df\|_{L^p(\R^d,w)}
\le C
\|f\|_{L^p(\R^d,\kM_dw)},
\quad p>1,
\end{equation}
for arbitrary $w\ge 0$?
\end{prob}

\noindent
Concerning Problem \eqref{prob1.1} 
it is known that, 
see \cite{Lin} (for $d=2$) and 
\cite{Pe} (for $d\ge 2$), 
\eqref{1.1} holds for all $p>1$ 
if $w\in A_{\8}^{*}$. 

We say that $w$ belongs to 
the class $A_p^{*}$ whenever
\begin{align*}
[w]_{A_p^{*}}
&=
\sup_{R\in\cR}
\fint_{R}w(x)\,dx\lt(\fint_{R}w(x)^{-1/(p-1)}\,dx\rt)^{p-1}
<\8,
\quad 1<p<\8,
\\
[w]_{A_1^{*}}
&=
\sup_{R\in\cR}
\frac{\fint_{R}w(x)\,dx}{\essinf_{x\in R}w(x)}
<\8.
\end{align*}
It follows by H\"{o}lder's inequality that 
the $A_p^{*}$ classes are increasing, 
that is, 
for $1\le p\le q<\8$ we have 
$A_p^{*}\subset A_q^{*}$. 
Thus one defines 
\[
A_{\8}^{*}
=
\bigcup_{p>1}A_p^{*}.
\]

The endpoint behavior of $\kM_d$ close to $L^1$ is given by 
Mitsis \cite{Mi} (for $d=2$) and 
Luque and Parissis \cite{LP} (for $d\ge 2$). 
That is, for $t>0$,
\[
w(\{x\in\R^d:\,\kM_df(x)>t\})
\le C
\int_{\R^d}
\frac{|f|}{t}
\lt(1+\lt(\log^{+}\frac{|f|}{t}\rt)^{d-1}\rt)
\kM_dw\,dx,
\]
holds for any $w\in A_{\8}^{*}$, 
where $\log^{+}t=\max(0,\log t)$. 

Concerning Problem \eqref{prob1.1} 
we established the following.

\begin{thm}[\cite{ST}]\label{thm1.2}
Let $w$ be any weight on $\R^2$ and 
set $W=\kM_2\kM w$. 
Then 
\[
w(\{x\in\R^2:\,\kM_2f(x)>t\})
\le C
\int_{\R^2}
\frac{|f|}{t}
\lt(1+\log^{+}\frac{|f|}{t}\rt)
W\,dx,\quad t>0,
\]
holds, where the constant $C>0$ does not depend on $w$ and $f$.
\end{thm}

In this paper 
we consider a weaker results of Theorem \ref{thm1.2} 
in the heigher dimensions (Theorem \ref{thm1.3}). 

Let $c=1,2,\ldots,d$. 
We say that the set of rectangles $R$ in $\R^d$ have the complexity $c$ 
whenever the sidelengths of $R$ are exactly 
$\al_1$ or $\al_2$ or $\ldots$ or $\al_c$ 
for varying $\al_1,\al_2,\ldots,\al_c>0$. 
That is, the set of rectangles with complexity $c$ 
is the $c$-parameter family of rectangles. 
For a locally integrable function $f$ on $\R^d$, 
we define the strong maximal operator $\kM_c$ by 
\[
\kM_cf(x)
=
\sup_{R\in\cR_c}
\1_{R}(x)\fint_{R}|f|\,dy,
\]
where $\cR_c$ is the set of all rectangles in $\R^d$ 
with sides parallel to the coordinate axes and 
having the complexity $c$. 
We notice that $\kM_1$ is the Hardy-Littlewood maximal operator $\kM$. 

\begin{thm}\label{thm1.3}
Let $c=1,2,\ldots,d$. 
Let $w$ be any weight on $\R^d$ and 
set $W=\kM_c\kM_{c-1}\cdots\kM_1w$. 
Then, for $p>1$, 
\[
w(\{x\in\R^d:\,\kM_cf(x)>t\})^{\frac1p}
\le
\frac{C}{t}
\|f\|_{L^p(\R^d,W)},
\quad t>0,
\]
holds, where the constant $C>0$ does not depend on $w$ and $f$.
\end{thm}

By interpolation, we have the following corollary.

\begin{cor}\label{cor1.4}
Let $c=1,2,\ldots,d$. 
Let $w$ be any weight on $\R^d$ and 
set $W=\kM_c\kM_{c-1}\cdots\kM_1w$. 
Then, for $p>1$, 
\[
\|\kM_cf\|_{L^p(\R^d,w)}
\le C
\|f\|_{L^p(\R^d,W)}
\]
holds, where the constant $C>0$ does not depend on $w$ and $f$.
\end{cor}

Probably, 
the endpoint Fefferman-Stein inequality 
for the strong maximal operator 
with compositions of some maximal operators 
hold in the heigher dimensions, 
but, 
we can not prove it until now. 
Further refinement of the known proofs for the boundedness of the strong maximal operator would be needed. 
In the last section, 
we will present an elementary proof of the endpoint estimate for the strong maximal operator 
(Theorem \ref{thm3.1}). 
Our method used is a covering lemma for rectangles due to 
Robert Fefferman and Jill Pipher \cite{FP}. 

The letter $C$ will be used for the positive finite constants that may change from one occurrence to another. 
Constants with subscripts, such as $C_1$, $C_2$, do not change
in different occurrences.

\section{Proof of Theorem \ref{thm1.3}}\label{sec2}
In what follows we shall prove Theorem \ref{thm1.3}. 
We denote by $P_i$, 
$i=1,2,\ldots,d$, 
the projection on the $x_i$-axis. 
First, we notice that the theorem holds for $c=1$. 
We assume that the theorem holds for $c=m-1$ 
and then we shall prove it for $c=m$. 
With a standard argument, 
we may assume that 
the basis $\cR_m$ is the set of all dyadic rectangles 
(cartesian products of dyadic intervals). 
By allowing a multiple constant $d!$, 
we further assume that, 
when $R\in\cR_m$, 
the sidelengths $|P_i(R)|$ 
decrease and 
\[
|P_1(R)|
=
|P_2(R)|
=\cdots=
|P_{\hat{m}}(R)|
>
|P_{\hat{m}+1}(R)|.
\]
Fix $t>0$ and 
given the finite collection of dyadic rectangles 
$\{R_i\}_{i=1}^{M}\subset\cR_m$ 
such that 
\begin{equation}\label{2.1}
\fint_{R_i}|f|\,dy>t,\quad
i=1,2,\ldots,M.
\end{equation}
It suffices to estimate 
$w\lt(\bigcup_{i=1}^{M}R_i\rt)$.

First relabel if necessary so that 
the $R_i$ are ordered so that 
their long sidelengths $|P_1(R_i)|$ decrease. 
We now give a selection procedure to find subcollection 
$\{\wt{R}_i\}_{i=1}^{N}
\subset
\{R_i\}_{i=1}^{M}$. 

Take $\wt{R}_1=R_1$. 
Suppose have now chosen the rectangles 
$\wt{R}_1,\wt{R}_2,\ldots,\wt{R}_{i-1}$. 
We select $\wt{R}_i$ to be the first rectangle $R_k$ 
occurring after $\wt{R}_{i-1}$ 
so that 
\[
\lt|\bigcup_{j=1}^{i-1}\wt{R}_j\cap R_k\rt|<\frac12|R_k|.
\]
Thus, we see that 
\begin{equation}\label{2.2}
\lt|\bigcup_{j=1}^{i-1}\wt{R}_j\cap\wt{R}_i\rt|<\frac12|\wt{R}_i|,
\quad i=2,3,\ldots,N.
\end{equation}

We claim that 
\begin{equation}\label{2.3}
\bigcup_{i=1}^{M}R_i
\subset
\lt\{
x\in\R^d:\,
\kM_{m-1}[\1_{\bigcup_{i=1}^{N}\wt{R}_i}](x)\ge\frac12
\rt\}.
\end{equation}
Indeed, 
choose any point $x$ inside a rectangle $R_j$ 
that is not one of the selected rectangles $\wt{R}_i$. 
Then, 
there exists a unique $J\le N$ such that 
\[
\lt|\bigcup_{i=1}^{J}\wt{R}_i\cap R_j\rt|\ge\frac12|R_j|.
\]
Since, 
$|P_l(\wt{R}_i)|\ge|P_l(R_j)|$ 
for $l=1,2,\ldots,\hat{m}$ 
and $i=1,2,\ldots,J$, 
we have that 
\[
P_l(\wt{R}_i)\cap P_l(R_j)
=
P_l(R_j)
\text{ when }
\wt{R}_i\cap R_j\ne\0.
\]
Thus, 
\begin{align*}
\bigcup_{i=1}^{J}\wt{R}_i\cap R_j
&=
\bigcup_{i=1}^{J}
\lt(\prod_{l=1}^{\hat{m}}P_l(R_j)\rt)
\times
\lt(\prod_{l=\hat{m}+1}^dP_l(\wt{R}_i)\cap P_l(R_j)\rt)
\\ &=
\lt(\prod_{l=1}^{\hat{m}}P_l(R_j)\rt)
\times
\bigcup_{i=1}^{J}
\lt(\prod_{l=\hat{m}+1}^dP_l(\wt{R}_i)\cap P_l(R_j)\rt).
\end{align*}
Hence, 
\[
\lt|
\bigcup_{i=1}^{J}
\lt(\prod_{l=\hat{m}+1}^dP_l(\wt{R}_i)\cap P_l(R_j)\rt)
\rt|
\ge\frac12
\lt|\prod_{l=\hat{m}+1}^dP_l(R_j)\rt|.
\]
Thanks to the fact that 
$|P_{\hat{m}+1}(R_j)|<|P_{\hat{m}}(R_j)|$,
this implies that 
\[
\lt|\bigcup_{i=1}^{J}\wt{R}_i\cap R\rt|
\ge\frac12|R|,
\]
where $R$ is a unique dyadic rectangle 
containing $x$ and satisfies 
\[
|P_1(R)|
=
|P_2(R)|
=\cdots=
|P_{\hat{m}}(R)|
=
|P_{\hat{m}+1}(R_j)|.
\]
This proves \eqref{2.3}, 
because such $R$ should belong to $\cR_{m-1}$. 

It follows from \eqref{2.3} and our assumption that 
\begin{align*}
w\lt(\bigcup_{i=1}^{M}R_i\rt)
&\le
w\lt(\lt\{
x\in\R^d:\,
\kM_{m-1}[\1_{\bigcup_{i=1}^{N}\wt{R}_i}](x)\ge\frac12
\rt\}\rt)
\\ &\le C
U\lt(\bigcup_{i=1}^{N}\wt{R}_i\rt)
\\ &\le C
\sum_{i=1}^{N}U(\wt{R}_i),
\end{align*}
where 
$U=\kM_{m-1}\kM_{m-2}\cdots\kM_1w$. 
We shall evaluate the quantity: 
\[
\text{(i)}
=
\sum_{i=1}^{N}U(\wt{R}_i).
\]

Set 
$E(\wt{R}_1)=\wt{R}_1$ and, 
for $i=2,3,\ldots,N$, 
\[
E(\wt{R}_i)
=
\wt{R}_i
\setminus
\bigcup_{j=1}^{i-1}\wt{R}_j.
\]
Then we have that 
the sets $E(\wt{R}_i)$ are pairwise disjoint and, 
by \eqref{2.2}, that 
\begin{equation}\label{2.4}
|E(\wt{R}_i)|\ge\frac12|\wt{R}_i|,
\quad i=1,2,\ldots,N.
\end{equation}

It follows from \eqref{2.1} that 
\begin{align*}
\text{(i)}
&\le
\frac1{t^p}
\sum_{i=1}^{N}
U(\wt{R}_i)
\lt(\fint_{\wt{R}_i}|f|\,dy\rt)^p
\\ &=
\frac1{t^p}
\sum_{i=1}^{N}
\lt(
\frac1{|\wt{R}_i|}
\int_{\wt{R}_i}|f|\,dy
\cdot
\lt(\fint_{\wt{R}_i}U\,dy\rt)^{\frac1p}
\rt)^p
|\wt{R}_i|
\\ &\le
\frac2{t^p}
\sum_{i=1}^{N}
\lt(
\fint_{\wt{R}_i}|f|W^{\frac1p}\,dy
\rt)^p
|E(\wt{R}_i)|
\\ &\le
\frac2{t^p}
\int_{\R^d}(\kM_m[fW^{\frac1p}])^p\,dx
\\ &\le
\frac{C}{t^p}
\int_{\R^d}|f|^pW\,dx,
\end{align*}
where we have used \eqref{2.4} 
(the sparseness property of $\{\wt{R}_i\}$) 
and the the $L^p$-boundedness of $\kM_m$. 
Altogether, 
this completes the proof. 

\section{An elementary proof of the endpoint estimate for the strong maximal operator}\label{sec3}
In what follows we shall prove the following theorem, 
which is originally due to in \cite{JMZ}. 

\begin{thm}\label{thm3.1}
Let $c=1,2,\ldots,d$. Then 
\[
|\{x\in\R^d:\,\kM_cf(x)>t\}|
\le C
\int_{\R^d}
\frac{|f|}{t}
\lt(1+\log^{+}\frac{|f|}{t}\rt)^{c-1}
\,dx,\quad t>0,
\]
holds, where the constant $C>0$ does not depend on $f$.
\end{thm}

\begin{proof}
First, we notice that the theorem holds for $c=1$. 
We assume that the theorem holds for $c=m-1$ 
and then we shall prove it for $c=m$. 
In the same manner as in the proof of Theorem \ref{thm1.3},
we may assume that 
the basis $\cR_m$ is the set of all dyadic rectangles and, 
if $R\in\cR_m$, then 
the sidelengths $|P_i(R)|$ decrease for $i$ and 
\[
|P_1(R)|
=
|P_2(R)|
=\cdots=
|P_{\hat{m}}(R)|
>
|P_{\hat{m}+1}(R)|.
\]
Fix $t>0$ and 
given the finite collection of dyadic rectangles 
$\{R_i\}_{i=1}^{M}\subset\cR_m$ 
such that 
\begin{equation}\label{3.1}
\fint_{R_i}|f|\,dy>t,\quad
i=1,2,\ldots,M.
\end{equation}
It suffices to estimate 
$\lt|\bigcup_{i=1}^{M}R_i\rt|$.

First relabel if necessary so that 
the $R_i$ are ordered so that 
their long sidelengths $|P_1(R_i)|$ decrease. 
We now give a selection procedure to find subcollection 
$\{\wt{R}_i\}_{i=1}^{N}
\subset
\{R_i\}_{i=1}^{M}$. 

Take $\wt{R}_1=R_1$. 
Suppose have now chosen the rectangles 
$\wt{R}_1,\wt{R}_2,\ldots,\wt{R}_{i-1}$. 
We select $\wt{R}_i$ to be the first rectangle $R_k$ 
occurring after $\wt{R}_{i-1}$ 
so that 
\[
\int_{R_k}
\exp\lt(\sum_{j=1}^{i-1}\1_{\wt{R}_j}\rt)^{\frac1{m-1}}
-1
\,dy
<e|R_k|.
\]
Thus, we see that 
\begin{equation}\label{3.2}
\int_{\wt{R}_i}
\exp\lt(\sum_{j=1}^{i-1}\1_{\wt{R}_j}\rt)^{\frac1{m-1}}
-1
\,dy
<e|\wt{R}_i|,
\quad i=2,3,\ldots,N.
\end{equation}

We claim that 
\begin{equation}\label{3.3}
\bigcup_{i=1}^{M}R_i
\subset
\lt\{
x\in\R^d:\,
\kM_{m-1}\lt[
\exp\lt(\sum_{i=1}^{N}\1_{\wt{R}_i}\rt)^{\frac1{m-1}}
-1
\rt](x)\ge e
\rt\}.
\end{equation}
Indeed, 
if $x$ belongs to some $\wt{R}_i$, 
it is of course obvious that it is contained on the set to the right, 
since then we can replace $\wt{R}_i$ by 
the rectangle $R$ which contains $x$, 
whose first $\hat{m}$ sidelengths are equals to 
the $(\hat{m}+1)$th sidelength of $\wt{R}_i$ 
and the other side lengths are remained. 
Choose any point $x$ inside a rectangle $R_j$ 
that is not one of the selected rectangles $\wt{R}_i$. 
Then, 
there exists a unique $J\le N$ such that 
\[
\int_{R_j}
\exp\lt(\sum_{i=1}^{J}\1_{\wt{R}_i}\rt)^{\frac1{m-1}}
-1
\,dy
\ge e|R_j|.
\]
Since, 
$|P_l(\wt{R}_i)|\ge|P_l(R_j)|$ 
for $l=1,2,\ldots,\hat{m}$ 
and $i=1,2,\ldots,J$, 
we have that 
\[
P_l(\wt{R}_i)\cap P_l(R_j)
=
P_l(R_j)
\text{ when }
\wt{R}_i\cap R_j\ne\0.
\]
It follows from Fubini's theorem that 
\begin{align*}
\lefteqn{
\int_{R_j}
\exp\lt(\sum_{i=1}^{J}\1_{\wt{R}_i}\rt)^{\frac1{m-1}}
-1
\,dy_1dy_2\cdots dy_d
}\\ &=
|P_1(R_j)|^{\hat{m}}
\int_{R_j'}
\exp\lt(\sum_{i=1}^{J}\1_{\wt{R}_i'}\rt)^{\frac1{m-1}}
-1
\,dy_{\hat{m}+1}dy_{\hat{m}+2}\cdots dy_d,
\end{align*}
where, for $R\in\cR_m$, 
\[
R'
=
\prod_{l=\hat{m}+1}^dP_l(R).
\]
Thus,
\[
\int_{R_j'}
\exp\lt(\sum_{i=1}^{J}\1_{\wt{R}_i'}\rt)^{\frac1{m-1}}
-1
\,dy_{\hat{m}+1}dy_{\hat{m}+2}\cdots dy_d
\ge e|R_j'|.
\]
Thanks to the fact that 
$|P_{\hat{m}+1}(R_j)|<|P_{\hat{m}}(R_j)|$, 
this implies that 
\[
\int_{R}
\exp\lt(\sum_{i=1}^{J}\1_{\wt{R}_i}\rt)^{\frac1{m-1}}
-1
\,dy
\ge e|R|,
\]
where $R$ is a unique dyadic rectangle 
containing $x$ and satisfies 
\[
|P_1(R)|
=
|P_2(R)|
=\cdots=
|P_{\hat{m}}(R)|
=
|P_{\hat{m}+1}(R_j)|.
\]
This proves \eqref{3.3}. 

It follows from \eqref{3.3} and our assumption that 
\begin{align*}
\lt|\bigcup_{i=1}^{M}R_i\rt|
&\le
\lt|\lt\{
x\in\R^d:\,
\kM_{m-1}\lt[
\exp\lt(\sum_{i=1}^{N}\1_{\wt{R}_i}\rt)^{\frac1{m-1}}
-1
\rt](x)\ge e
\rt\}\rt|
\\ &\le C
\int_{\R^d}
\lt(\exp\lt(\sum_{i=1}^{N}\1_{\wt{R}_i}\rt)^{\frac1{m-1}}-1\rt)
\lt(\sum_{i=1}^{N}\1_{\wt{R}_i}\rt)^{\frac{m-2}{m-1}}
\,dx
\\ &=C
\sum_{k=1}^{\8}
\frac1{k!}
\int_{\R^d}
\lt(\sum_{i=1}^{N}\1_{\wt{R}_i}\rt)^{\frac{k+m-2}{m-1}}
\,dx.
\end{align*}
We use an elementary inequality:
\begin{equation}\label{3.4}
\lt(\sum_{i=1}^{\8}a_i\rt)^s
\le s
\sum_{i=1}^{\8}
a_i
\lt(\sum_{j=1}^ia_j\rt)^{s-1},
\quad s>1,
\end{equation}
where $\{a_i\}_{i=1}^{\8}$ 
is a sequence of summable nonnegative reals. 
Then, for $k>1$, 
\begin{align*}
\lefteqn{
\int_{\R^d}
\lt(\sum_{i=1}^{N}\1_{\wt{R}_i}\rt)^{\frac{k+m-2}{m-1}}
\,dx
}\\ &\le
\frac{k+m-2}{m-1}
\sum_{i=1}^{N}
\int_{\wt{R}_i}
\lt(\sum_{j=1}^i\1_{\wt{R}_j}\rt)^{\frac{k-1}{m-1}}
\,dx
\\ &\le k
\sum_{i=1}^{N}
\int_{\wt{R}_i}
\lt(\sum_{j=1}^i\1_{\wt{R}_j}\rt)^{\frac{k-1}{m-1}}
\,dx.
\end{align*}
Inserting this estimate and changing the order of sums, 
we obtain 
\begin{align*}
\lt|\bigcup_{i=1}^{M}R_i\rt|
&\le C
\int_{\R^d}
\lt(\sum_{i=1}^{N}\1_{\wt{R}_i}\rt)
\,dx+C
\sum_{i=1}^{N}
\int_{\wt{R}_i}
\exp\lt(\sum_{j=1}^i\1_{\wt{R}_j}\rt)^{\frac1{m-1}}
-1
\,dx
\\ &\le C
\sum_{i=1}^{N}|\wt{R}_i|,
\end{align*}
where we have used \eqref{3.2} and 
\begin{align}\label{3.5}
\lefteqn{
\int_{\wt{R}_i}
\exp\lt(\sum_{j=1}^i\1_{\wt{R}_j}\rt)^{\frac1{m-1}}
-1
\,dx
}\\ \nonumber &\le
\int_{\wt{R}_i}
\exp\lt[\lt(\sum_{j=1}^{i-1}\1_{\wt{R}_j}\rt)^{\frac1{m-1}}+1\rt]
-1
\,dx
\\ \nonumber &\le e
\int_{\wt{R}_i}
\exp\lt(\sum_{j=1}^{i-1}\1_{\wt{R}_j}\rt)^{\frac1{m-1}}
-1
\,dx
+
(e-1)|\wt{R}_i|.
\end{align}
We shall evaluate the quantity: 
\[
\text{(i)}
=
\sum_{i=1}^{N}|\wt{R}_i|.
\]

By \eqref{3.1} we have that 
\begin{align*}
\text{(i)}
&\le
\sum_{i=1}^{N}
\int_{\wt{R}_i}\frac{|f|}{t}\,dy
\\ &=
\int_{\R^d}
\lt(\sum_{i=1}^{N}\1_{\wt{R}_i}\rt)
\cdot
\frac{|f|}{t}
\,dx.
\end{align*}

We now employ the following inequality: 
For $a\ge 0$, let 
\[
\phi(a)
=
\int_0^a\exp s^{\frac1{m-1}}\,ds.
\]
Then we easily see that, 
by noticing $\phi(a)>a$, 
\[
ab\le\phi(a)+b(\log^{+}b)^{m-1},
\quad a,b>0.
\]

Choosing $\dl_0$ small enough determined later, 
we obtain 
\begin{align*}
\text{(i)}
&\le\dl_0
\int_{\R^d}
\phi\lt(\sum_{i=1}^{N}\1_{\wt{R}_i}\rt)
\,dx
\\ &\quad+
\int_{\R^d}
\frac{|f|}{t}
\lt(1+\log^{+}\frac{|f|}{\dl_0t}\rt)^{m-1}
\,dx.
\end{align*}
We have to evaluate the quantity: 
\[
\text{(ii)}
=
\int_{\R^d}
\phi\lt(\sum_{i=1}^{N}\1_{\wt{R}_i}\rt)
\,dx.
\]

There holds 
\begin{align*}
\phi(a)
&=
\int_0^a\exp s^{\frac1{m-1}}\,ds
\\ &=
\sum_{k=0}^{\8}
\frac{1}{k!}
\int_0^a s^{\frac{k}{m-1}}\,ds
\\ &=
a+\sum_{k=1}^{\8}
\frac{m-1}{(k+m-1)k!}
a^{\frac{k+m-1}{m-1}},
\end{align*}
which entails 
\[
\text{(ii)}
=
\int_{\R^d}
\lt(\sum_{i=1}^{N}\1_{\wt{R}_i}\rt)
\,dx
+
\sum_{k=1}^{\8}
\frac{m-1}{(k+m-1)k!}
\int_{\R^d}
\lt(\sum_{i=1}^{N}\1_{\wt{R}_i}\rt)^{\frac{k+m-1}{m-1}}
\,dx.
\]
It follows from \eqref{3.4} that, 
for $k>0$, 
\begin{align*}
\lefteqn{
\frac{m-1}{(k+m-1)k!}
\int_{\R^d}
\lt(\sum_{i=1}^{N}\1_{\wt{R}_i}\rt)^{\frac{k+m-1}{m-1}}
\,dx
}\\ &\le
\frac1{k!}
\sum_{i=1}^{N}
\int_{\wt{R}_i}
\lt(\sum_{j=1}^i\1_{\wt{R}_j}\rt)^{\frac{k}{m-1}}
\,dx.
\end{align*}
Inserting this estimate and changing the order of sums, 
we obtain 
\begin{align*}
\text{(ii)}
&\le 
\int_{\R^d}
\lt(\sum_{i=1}^{N}\1_{\wt{R}_i}\rt)
\,dx
+
\sum_{i=1}^{N}
\int_{\wt{R}_i}
\exp\lt(\sum_{j=1}^i\1_{\wt{R}_j}\rt)^{\frac1{m-1}}
-1
\,dx
\\ &\le C_0
\text{(i)},
\end{align*}
where we have used \eqref{3.2} and \eqref{3.5}. 

If we choose $\dl_0$ so that 
$C_0\dl_0=\frac12$,
we obtain 
\[
\text{(i)}
\le C
\int_{\R^d}
\frac{|f|}{t}
\lt(1+\log^{+}\frac{|f|}{t}\rt)^{m-1}
\,dx.
\]
This completes the proof. 
\end{proof}

\end{document}